\documentclass[preprint,12pt]{elsarticle}

\usepackage{graphicx}

\usepackage{amssymb}

\journal{Journal of Differential Equations}

\usepackage{amsmath}
\newtheorem{theorem}{Theorem}[section]
\newtheorem{lemma}{Lemma}[section]
\newtheorem{proposition}{Proposition}[section]
\newtheorem{corr}{Corollary}[section]
\newtheorem{rem}{Remark}[section]
\newtheorem{defn}{Definition}[section]

\newcommand{\rr}{\mathbb R}
\newcommand{\bN}{\mathbb N}
\newcommand{\bR}{\mathbb R}
\newcommand{\Int}{\mathop{\mathrm{Int}}\nolimits}

\newcommand{\Fr}{\mathop{\mathrm{Fr}}\nolimits}
\newcommand{\grad}{\mathop{\mathrm{grad}}\nolimits}
\newcommand{\diam}{\mathop{\mathrm{diam}}\nolimits}
\newcommand{\Cl}[1]{\mathop{\overline{#1}}\nolimits}
\renewcommand{\emptyset}{\varnothing}

\newcommand{\myvectorfield}{\overrightarrow{f}}
\if\csname proof\endcsname\relax
\newenvironment{proof}{\noindent{\bf Proof.\,}}{\hfill $\Box$\par\medskip}
\fi

\begin{document}

\begin{frontmatter}



\title{Discrete conditions of Lyapunov stability\tnoteref{support1}}
\tnotetext[support1]{The research was partially supported by DKNII (M/377-2012) and DFFD (F40.1/009).}


\author[Imath]{E.~Polulyakh}
\ead{polulyah@imath.kiev.ua}
\author[Imath]{V.~Sharko}
\ead{sharko@imath.kiev.ua}
\author[Imath]{I.~Vlasenko}
\ead{vlasenko@imath.kiev.ua}

\address[Imath]{Inst. of Math., NAS of Ukraine, Kiev}

\begin{abstract}
We address the classic
problem of stability and asymptotic stability in the
sense of Lyapunov of the equilibrium point of autonomic differential
equations using discrete approach. This new approach includes a
consideration of a family of hypersurfaces instead of the Lyapunov
functions, and conditions on the right part of the differential
equation instead of conditions on a Lyapunov function along
trajectories of the equation.

In this paper we generalize results of~\cite{SharkoYu2005,SharkoYu2010}.
\end{abstract}


\end{frontmatter}



\section {Introduction.}

Consider a system of differential equations
\begin{equation}\label{eq_system_1}
d{\bf x}/dt = \myvectorfield({\bf x}), \;\;\;\; x(0) = x_0,
\end{equation}
which is defined in a neighbourhood $G$ of $x_0$ from $\bR^n$ and such that $x_0$ is its equilibrium point.

Suppose that $\myvectorfield({\bf x})$ is a $C^1$-smooth function.
Denote by $x_p(t)$ the solution of the system~\eqref{eq_system_1} where
$p$ is a point such that $x_p(0)=p$.

\begin{defn}\label{defn_lse}
The equilibrium $x_o$ of the above system is said to be \emph{Lyapunov stable}, if, for every $\epsilon > 0$, there exists $\delta = \delta(\epsilon) > 0$ such that, if $\|x_p(0)-x_0\| < \delta$, then $\|x_p(t)-x_o\| < \epsilon$, for every $t \geq 0$.
%
\end{defn}

The classic method to prove that an equilibrium of the system~\eqref{eq_system_1} is Lyapunov stable
is to build a Lyapunov function for that system, \cite{RushAbetsLalua80}.

We present a new practical method of proving the Lyapunov stability of a equilibrium point
using a sequence of nested hypersurfaces.

A Lyapunov function always exists
in the case when the equilibrium is asymptotically stable,
see~\cite{Wilson67}, or in the case of orbital stability, see~\cite{Sharkovsky70}.
However, it is not the case
when the equilibrium is Lyapunov stable but not asymptotically stable.


The paper~\cite{Polulyakh2012} presents an example of a dynamical
system such that its critical point is Lyapunov stable but
no Lyapunov function exist in a neighbourhood of that critical point.
Still, the method of sequences of nested hypersurfaces presented in
this paper works for the system in~\cite{Polulyakh2012}.

Sequences of nested hypersurfaces can be naturally viewed
as a generalization of Lyapunov functions. A Lyapunov funcfion naturally provides
a continuous foliation of its level surfaces and scalar product with its gradient vector field.
A countable subsequence of level surfaces naturally act as a
sequence of nested hypersurfaces in proposed method.
To further provide a link between Lyapunov functions and sequences of nested hypersurfaces
we define $L$-functions which act as generators of sequences of nested hypersurfaces
and use them to prove the existence theorem for the sequences of nested hypersurfaces
and study stability of critical points of gradient systems.

The approach that uses a discrete sequence of nested hypersurfaces instead of a Lyapunov function
was introduced in the papers~\cite{SharkoYu2005,SharkoYu2010}.
Our article generalizes the results of these papers.

Authors want to express their gratitude to A.~N. Sharkovsky for his
valuable remarks and to S.~I. Maksimenko for his help with the article.

\section{Sequences of converging nested hypersurfaces.}

Denote by $\rr^n$ an $n$-dimensional Euclidian space. Let $\rho$ be the standard metrics on this space. For a bounded set $A \subset \bR^n$ we write
\[
\diam(A) = \sup_{\mathbf{x}, \mathbf{y} \in A} \rho(\mathbf{x}, \mathbf{y}) \,.
\]

\begin{defn}\label{defn_surf}
Let $H^{n-1} \subset \bR^n$ be a connected closed hypersurface (smooth compact submanifold of dimension $n-1$ which has empty boundary).
Let us say that $H^{n-1}$ {\bf bounds} a point $p$ in $\bR^n$ if $p \notin H^{n-1}$ and any path $\gamma$  from $p$
to $\mathbf{x} \in \bR^n$ intersects $H^{n-1}$ when $\rho(p, \mathbf{x}) > \diam(H^{n-1})$.
\end{defn}

Note that a connected closed hypersurface in the Euclidian space is always oriented and splits the space on two components. One of them is bounded, and the other is not \cite{two_components}. Let us call the bounded component of the complement the \emph{internal component}. It follows that a hypersurface has two different normal vector fields of unit length, one of which is directed towards the internal component, and the other is directed towards the other component.

\begin{defn}
\label{defn_nested_seq}
We call a sequence of connected closed hypersurfaces $H^{n-1}_i$ \textbf{nested} if
$H^{n-1}_{i+1}$ is contained in the internal component of the complement $\bR^n \setminus H^{n-1}_{i}$
for every $i \in \bN$.
\end{defn}

Let $H^{n-1} \subset \bR^n$ be a hypersurface that bounds a point $p$ in $\bR^n$.
Denote by
\[
d (p, H^{n-1}) = \max_{y\in H^{n-1}} \left \{||p-y||\right\}
\]
a Hausdorff distance between a point $p$ and $H^{n-1}$.

\begin{defn}
\label{defn_surf_family_converge}
A sequence of nested hypersurfaces $\{H^{n-1}_i\}$ such that each $H^{n-1}_i$ bounds a point $p$ is said \textbf{to converge to} $p$ in $\bR^n$ if $d (p, H^{n-1}_i)\to 0$ as $i \to \infty$.
\end{defn}

Let us consider an autonomous system of differential equations~\eqref{eq_system_1}.
For the sake of convenience choose new coordinates that make the equilibrium point $x_0$ the origin.

By $\{{\bf H}_i^{n-1}\}$ denote a sequence of nested hypersurfaces that converge to origin.
Let $\vec{N_i}({\bf x})$ be the vector fields of unit length on each
${\bf H}_i^{n-1}$ such that vectors of $\vec{N_i}({\bf x})$ are
orthogonal to ${\bf H}_i^{n-1}$ and direct towards the internal component.
Naturally, the system~\eqref{eq_system_1} generates the smooth vector field $\myvectorfield({\bf x})$
on each $\{{\bf H}_i^{n-1}\}$.
Denote by $S_i({\bf x})$ the scalar product $<\vec{N_i}({\bf x}), \myvectorfield({\bf x})>$.
The function $S_i({\bf x})$ is defined for each hypersurface ${\bf H}_i^{n-1}$ and
shows how integral trajectories of the system~\eqref{eq_system_1} intersect ${\bf H}_i^{n-1}$.

We prove here the following theorem.
\begin{theorem}\label{th:lyapunov_from_ge0}
If there exists a sequence $\{{\bf H}_i^{n-1}\}$ such that $\forall i$
we have $ S_i({\bf x}) \ge 0$
then the origin is stable in the sense of Lyapunov for the system~\eqref{eq_system_1}.
\end{theorem}

The proof of Theorem~\ref{th:lyapunov_from_ge0} is given in next section.

This theorem is a strong version of results in~\cite{SharkoYu2005,SharkoYu2010},
where the approach to study stability using a discrete set of nested
hypersurfaces was first used.

Let us recall some definitions.

A \emph{regular value} of a smooth function $F$ is a value such that the differential of $F$ is non-zero in every preimage of this value.

The maximal connected subsets (ordered by inclusion) of a nonempty topological space are called the \emph{connected components} of the space.

\begin{defn}\label{defn_quasi_isolated-point}
Suppose $z = F({\bf x})$  is a continuous function defined in a domain $G\subset \rr^n$. A point  ${\bf y} \in G$ is \emph{quasi-isolated} for the function  $z = F({\bf x})$, if  $\{{\bf y}\}$ is a connected component of the set  $F^{-1}(F({\bf y}))$.
\end{defn}

\begin{proposition}\label{proposition:local extremum}
Let $z = F({\bf x})$ be a continuous function defined in a domain   $G \subset \rr^n (n\geq 2)$ and ${\bf y} \in G$. Then ${\bf y} $ is a local  maximum or local minimum for $F$ if and only if $y$ is an isolated point of the set $F^{-1}(F({\bf y}))$.
\end{proposition}

{\it Proof.} Necessity is obvious.

Sufficiency is a consequence of the following arguments. Consider a domain $U \subset G $, such that  $F^{-1}(F({\bf y}))\cap U = {\bf y}$.  Let  ${\bf y_1} $ and ${\bf y_2} $ be two points in $G $ such that $ F({\bf y_1})  >  F({\bf y}) >  F({\bf y_2})$. Consider  a continuous path $\gamma (t)$ in $U $, such that $\gamma (0) = {\bf y_2}$,  $\gamma (1) = {\bf y_1}$ and $\gamma(t) \cap {\bf y} = \emptyset$.
Let $ F({\gamma (t)})$ be the restriction of the function   $z = F({\bf x})$ to the path $\gamma (t)$. Since $ F(\gamma (1)) > F(\gamma (0))$ and the function $F(\gamma (t))$ is continuous, then there must be a $t_0$, such that  $ F({\gamma (t_0)})= F({\bf y})$. But this is impossible due to the selection of the path $\gamma (t)$.

Therefore for any point ${\bf y_1}\in U$ we have that either $ F({\bf y_1})  >  F({\bf y})$ and so $\mathbf{y}$ is a local minimum, or $F({\bf y_1})  <  F({\bf y})$ and then $\mathbf{y}$ is a local maximum.
$\square$

\begin{proposition}\label{proposition:quasi_isolated-point is critical point}
Suppose $z = F({\bf x})$  is a $C^r$-smooth function defined in the domain  $G\subset \rr^n$ and point  ${\bf y} \in G$ is quasi-isolated for the function  $z = F({\bf x})$.  Then ${\bf y}$  is a  critical point of the function  $z = F({\bf x})$.
\end{proposition}

{\it Proof.} Suppose $z$ is not critical for $F$. Then there exist change of coordinates in a neighborhood $U \cap G$  of the point ${\bf y} $, such that our function will be linear function in $U$. But this contradicts to the assumption ${\bf y} $ is a quasi-isolated point.
$\square$

\begin{defn}\label{defn_L_function}
Suppose $z = F({\bf x})$  is a $C^r$-smooth function defined in a domain $ G\subset \rr^n$ and  ${\bf y} \in G$. Let $ F({\bf y}) = a$.
The function $z = F({\bf x})$ is called an $\bf {L}$-function for the point ${\bf y}$ if
there exists a sequence $(a_i)$ of regular values of $z = F({\bf x})$
with the following properties:
\begin{itemize}
\item [i.] $a_i \rightarrow a$ when $i \rightarrow \infty$;
\item [ii.] for each $i$ there exists a connected component ${\bf
    H}_{i}^{n-1}$ of the set $F^{-1}(a_i)$ such that ${\bf H}_{i}^{n-1}$ is a smooth hypersurface that bounds the point  ${\bf y}$;
\item [iii.] diameters of ${\bf H}_{i}^{n-1}$ tend to $0$ when $i \rightarrow \infty$.
\end{itemize}
\end{defn}

\begin{proposition}\label{proposition: critical point of L}
Let $F$ be an $\bf {L}$-function for a point ${\bf y} \in G$. Then ${\bf y}$ is the critical point of the function $z = F({\bf x})$.
\end{proposition}
{\it Proof.} Suppose the opposite. Let the point  ${\bf y}$ be a regular  point of the function $z = F({\bf x})$. Then there exist change of coordinates in a neighborhood $U$ of the point ${\bf y} $,
such that our function will be linear function in $U$. This contradicts to the existence of
smooth hypersurfaces ${\bf H}_{i}^{n-1}$  that bound the point  ${\bf y}$.
$\square$

Note that if $z = F({\bf x})$ is a smooth $\bf {L}$-function, than it
defines not only the sequence of hypersurfaces ${\bf H}_{i}^{n-1}$ but
also a gradient vector field $\overrightarrow{\grad}F({\bf x})$ which
can act as an orthogonal vector field $\vec{N_i}({\bf x})$ in the
definition of the function $S_i({\bf x})$ introduced above.
On each $H^{n-1}_i$ all vectors $\overrightarrow{\grad}F({\bf x})$ are either directed towards the internal component or they have the opposite direction. Let us assign $\varepsilon(H^{n-1}_i) = +1$ if  $\overrightarrow{\grad}F({\bf x})$ is directed towards the internal component on $H^{n-1}_i$ and $\varepsilon(H^{n-1}_i) = -1$ otherwise.

By $\tilde S_i({\bf x})$ denote the following function
$<\varepsilon(H^{n-1}_i) \overrightarrow{\grad}F({\bf x}), \myvectorfield({\bf x})>$.

\begin{corr}
Let $z = F({\bf x})$ be an $\bf {L}$-function.
Fix a sequence of hypersurfaces $\{ {\bf H}_{i}^{n-1} \}$ satisfying Definition~\ref{defn_L_function}.

If $\tilde S_i({\bf x})\ge 0$ for any $i$ in every point ${\bf x}$ of submanifold ${\bf H}_{i}^{n-1}\in \{ {\bf H}_{i}^{n-1} \}$
then the origin is stable in the sense of Lyapunov for the system~\eqref{eq_system_1}.
\end{corr}

{\it Proof.}
Observe that $\varepsilon(H^{n-1}_i) \overrightarrow{\grad}F({\bf x})$ is the normal vector field on ${\bf H}_{i}^{n-1}$ and is directed towards its internal component for every $i$. Stability in the sense of Lyapunov for the system~\eqref{eq_system_1} is the straightforward consequence of Theorem~\ref{th:lyapunov_from_ge0}.
$\square$

\section{Proof of theorem~\ref{th:lyapunov_from_ge0}.}\label{sec:lyapunov_from_ge0}

In order to prove theorem~\ref{th:lyapunov_from_ge0} we first need to prove some auxiliary statements.

Suppose $K^n \subset \rr^n$ is a compact manifold with a $C^1$-smooth  closed boundary hypersurface ${\bf H}^{n-1}$ and with an interior $W = \Int K^n$.
Let $\vec{N}({\bf x})$ be the unit normal vector field on ${\bf H}^{n-1}$ directed towards interior of $K^n$.
Let also $\myvectorfield({\bf x})$ be a vector field defined in a neighbourhood of $K^n$ such that $S({\bf x})\ge 0$ on ${\bf H}^{n-1}$, where $S({\bf x})$ is the scalar product $<\vec{N}({\bf x}), \myvectorfield({\bf x})>$.

According to Long Tubular Flow Theorem (see \cite{PalisDiMelo}), for any arc of a trajectory of the vector field
$\myvectorfield({\bf x})$
which is compact and not closed there exists a $C^1$-smooth flow-box containing that arc.
Consider a flow-box of an arc of a trajectory of the vector field $\myvectorfield({\bf x})$.
The boundary of the flow-box consists of three parts:
two bases (parts of the boundary that act as cross-sections of the the flow)
and the side part that consists of flow lines.

\begin{figure}[htbp]
  \centering
\includegraphics[scale=1.0]{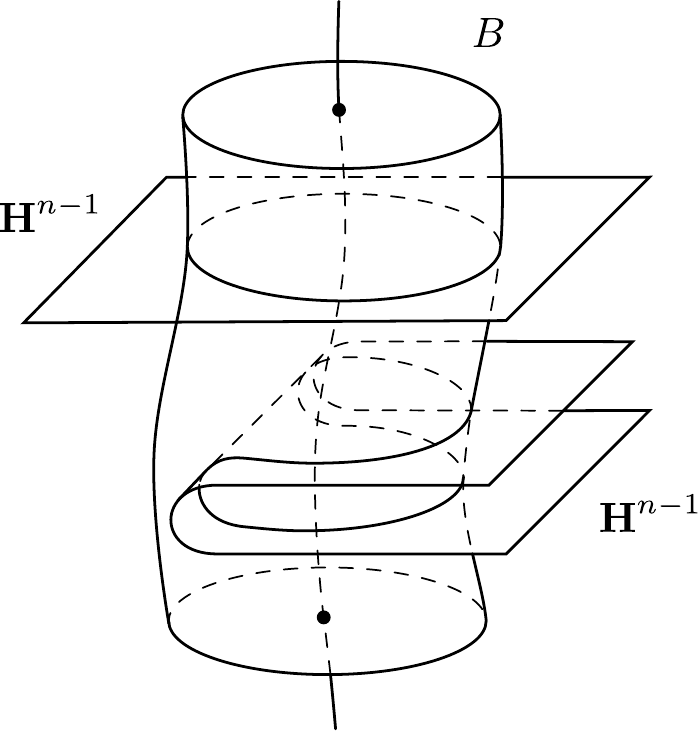}
  \caption{a flow-box $B$ that intersects ${\bf H}^{n-1}$.}
  \label{fig:flowbox}
\end{figure}

\begin{lemma}\label{lm:flow_box_measure_0}
Let $B$ be a flow-box of an arc of a trajectory of the vector field $\myvectorfield({\bf x})$.
Suppose that $B$ intersects ${\bf H}^{n-1}$ so that the bases of $B$ does not intersect ${\bf H}^{n-1}$.
(See fig.~\ref{fig:flowbox}).
Let $T$ be the set of points of the hypersurface ${\bf H}^{n-1}$ where flow lines are  tangent to $H^{n-1}$.
Let also $p_1$ and $p_2$ be projections of $T$ along the flow lines
on the bases of $B$. Then the Lebesgue measures of $T_1 = p_1(T)$, $T_2 = p_2(T)$ in the corresponding bases are 0.
\end{lemma}

{\it Proof.}
We can completely ignore the tangent points that belong to the part of
the boundary of $B$ that consists of flow lines because projections of
those points on the cross-section bases of $B$ have zero Lebesgue measure.
Consider the set $T_0 = T \cap \Int B$.
The intersection of $\Int B$ and the hypersurface ${\bf H}^{n-1}$ is open in ${\bf H}^{n-1}$.
Hence, the intersection $\Int P\cap {\bf H}^{n-1}$ is a submanifold
in ${\bf H}^{n-1}$.
Projection $p_1$ of the set $T_0$ along flow lines on a cross-section
base is a smooth map. The smoothness of $p_1$ is the same as the
smoothness of the flow-box $B$. Therefore, the set of singularities of
$p_1$ has zero Lebesgue measure according to the Sard's Theorem~\cite{Sternberg}.
The same arguments hold for $p_2$.

It is obvious that the set of  singularities of $p_1$ (respectively, of $p_2$) coinsides with the intersection of the set of tangent points of hypersurface ${\bf H}^{n-1}$ to the flow lines with the interior of the flow-box.
$\square$

\begin{lemma}\label{th:lyapunov_boundness_surface}

If $S({\bf x})\ge 0$ on ${\bf H}^{n-1}$ then any trajectory of the
vector field $\myvectorfield({\bf x})$ does not leave the manifold $K^n$.
\end{lemma}

Assume the converse.
Then there exists an integral trajectory
$\xi$ of the vector field $\myvectorfield({\bf x})$ such that $\xi$
leaves $K^n$. In other words,
$\xi \cap K^n\not=\emptyset$ and $\xi \cap \left( \rr^n \setminus K^n \right)\not=\emptyset$.
Since ${\bf H}^{n-1}$ is a boundary of $K^n$, $\xi \cap {\bf H}^{n-1}\not=\emptyset$ as well.
Note that $\xi$ can not be an equilibrium point of the vector field
$\myvectorfield({\bf x})$, because in this case $\xi$ is just a point
and can go nowhere.
Therefore, $\myvectorfield({\bf x})$ is a non-zero vector field along
$\xi$ and in a neighbourhood of $\xi$.

Choose a flow-box $B$ of the trajectory $\xi$ such that the ``out'' base of $B$ (i.~e. the part of $\partial P$ the trajectories of $B$ are going out through) does not intersect ${\bf H}^{n-1}$.
It is always possible because $\xi$ leaves $K^n$.
Denote the ``out'' base of $B$ by $B_{\rm\textsc{out}}$ and the ``in'' base of $B$ by $B_{\rm\textsc{in}}$.
Then, the ``out'' base $B_{\rm\textsc{out}}$ must be outside $K^n$.
(See fig.~\ref{fig:lemma}).

\begin{figure}[htbp]
  \centering
\includegraphics[scale=1.0]{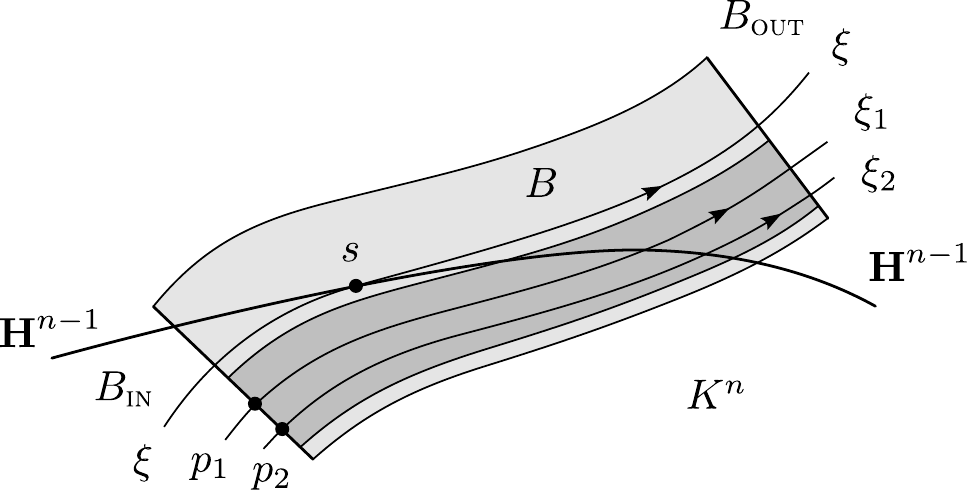}
  \caption{Illustration to the proof of Lemma\ref{th:lyapunov_boundness_surface}.}
  \label{fig:lemma}
\end{figure}

Since $\myvectorfield({\bf x})$ is a non-zero vector field along $\xi$, then $\xi$ should travel some time either in $W$ or on ${\bf H}^{n-1}$ before leaving $K^n$.
In the former case, $B_{\rm\textsc{in}}$ can be chosen to be inside of $W$ similarly to $B_{\rm\textsc{out}}$.
Consider the latter case. Pick a point $s\in \xi \cap {\bf H}^{n-1}$.
Choose $B$ to be long enough to contain $s$ inside.
$s\in {\bf H}^{n-1}$ and ${\bf H}^{n-1}=\partial K^n$. Hence, in any
neighbourhood of $s$ there is an open set of points that belongs to
$W=\Int K$. But $\Int B$ is also the neighbourhood of $s$ in $\rr^n$.
Therefore, we can always adjust $B_{\rm\textsc{in}}$ to contain the open in
$B_{\rm\textsc{in}}$ set of points that belong to $W$.

As a result, in both cases $B_{\rm\textsc{in}}$ can be at least chosen to
have an open (in $B_{\rm\textsc{in}}$) non-empty intersection with $W$.

Since the intersection $B_{\rm\textsc{in}} \cap W$ is open in $B_{\rm\textsc{in}}$
there exists a point $p_1\in B_{\rm\textsc{in}}$ such that a whole neighbourhood
$V(p_1)\ni p_1$ is also contained in $B_{\rm\textsc{in}}\cap W$.
Choose a new flow-box $B^1\subset B$  for the trajectory $\xi_1$ of the point $p_1$
with bases $B^1_{\rm\textsc{in}}$ and $B^1_{\rm\textsc{out}}$
such that $B^1_{\rm\textsc{in}}\subset V(p_1)\subset B_{\rm\textsc{in}}\cap W$ and
$B^1_{\rm\textsc{out}} \subset B_{\rm\textsc{out}}$.
By construction, both bases of $B^1$ do not intersect ${\bf H}^{n-1}$.

Applying Lemma~\ref{lm:flow_box_measure_0}, we conclude that the set $T
\subset B^1_{\rm\textsc{in}}$ of points whose trajectories always intersect ${\bf H}^{n-1}$
transversally have the full Lebesgue measure in $B^1_{\rm\textsc{in}}$.
Consider a point $p_2\in T$ such that the trajectory $\xi_2$ of
$p_2$ always intersects ${\bf H}^{n-1}$ transversally.
By construction, $\xi_2$ intersects $B^1_{\rm\textsc{in}}$ inside of $W$ and
intersects $B^1_{\rm\textsc{out}}$ outside of $K$. Also, $\xi_2$ intersects
${\bf H}^{n-1}$ in a finite number of points because the intersection is
always transversal. In those points $S({\bf x})\not= 0$ due to transversality.
But the hypersurface ${\bf H}^{n-1}$ divides $\rr^n$.
Then $\xi_2$ can only leave $K$ at a point of intersection with ${\bf
 H}^{n-1}$. Furthermore, this point of intersection is transversal due
to the choice of $\xi_2$. But then $S(p_2)< 0$ that contradicts the
conditions of the theorem.
This contradiction proves the theorem.$\square$ 

Using Lemma~\ref{th:lyapunov_boundness_surface} we can finish the proof of Theorem~\ref{th:lyapunov_from_ge0}.

\textbf{Proof.}
By assumption, $ S({\bf x})\ge 0$ on every hypersurface ${\bf H}_{i}^{n-1}$.
According to Lemma~\ref{th:lyapunov_boundness_surface},
for every hypersurface ${\bf H}_{i}^{n-1}$ any trajectory ${\bf
 X}(t)$, $t\ge 0$, of the system~\eqref{eq_system_1} that begins in a
point ${\bf x}\in {\bf H}_{i}^{n-1}$ does not leave the manifold
$K_{i}$ whose boundary is ${\bf H}_{i}^{n-1}$.
As a consequence, any other trajectory that starts at a point of
$K_{i}$ can not leave $K_{i}$, because otherwise that trajectory would
intersect ${\bf H}_{i}^{n-1}$.

As diameters of ${\bf H}_{i}^{n-1}$ tend to $0$ when $i \rightarrow \infty$
then stability of the origin in the sense of Lyapunov directly follows
from the definition of Lyapunov stability.
$\square$

\section {On existence of sequences of nested hypersurfaces.}

The theorem below shows how to build sequences of nested hypersurfaces converging to a point using a smooth enough function on an open domain $G \subset \bR^n$.

\begin{theorem}\label{th_surf_family}
Let $G$ be a domain in $\bR^n$. Assume that $F \in C^n(G)$.

Then $F$ is an $\bf{L}$-function  for $x_0 \in G$ if and only if $x_0$ is a quasi-isolated point of $F$.
\end{theorem}

The proof of this theorem is given in section~\ref{sec:th_surf_family_proof}.

\begin{corr}\label{corr:non-isolated_crit_point}
Suppose $z = F({\bf x})$  is a $C^n$-smooth function defined in a domain  $G\subset \rr^n$. Denote by $\Sigma(F({\bf x}))$ the set of critical points of the function $z = F({\bf x})$.

Let ${\bf y} \in G$ be a quasi-isolated point for the function  $z = F({\bf x})$. Assume also that $\mathbf{y}$ is not an isolated point of the level set $F^{-1}(F({\bf y}))$.

Then $\mathbf{y} \in \Sigma(F({\bf x}))$ and $\mathbf{y}$ is not an isolated point of the set $\Sigma(F({\bf x}))$.
\end{corr}

{\it Proof.}
Observe that $\mathbf{y}$ is a critical point of $F$ according to Proposition~\ref{proposition:quasi_isolated-point is critical point}.

Let $\mathbf{y}$ be an isolated critical  point. Then there exists an $\varepsilon > 0$ such that the open ball $U = \{ \mathbf{x} \in \bR^n \,|\, \rho(\mathbf{x}, \mathbf{y}) < \varepsilon \} \subset G$ does not contain other critical points of $F$.

Since $\mathbf{y}$ is quasi-isolated, it follows from Theorem~\ref{th_surf_family} that $F$ is an $\bf{L}$-function  for $\mathbf{y}$. Hence there exist a regular value $a$ of $F$ and a hypersurface $H^{n-1} \subset U \cap F^{-1}(a)$, such that $\mathbf{y}$ is contained in the inner component $W$ of $\bR^n \setminus H^{n-1}$. It is easy to see that $\Cl{W} \subset U$. Observe also that $\mathbf{y} \notin F^{-1}(a)$	 because $\mathbf{y}$ is critical point of $F$ and can not be contained in a regular level set.

The compact set $\Cl{W}$ has the interior $W$ and the frontier $H^{n-1}$. The function $F$ is continuous on $\Cl{W}$, so it achieves its maximum and minimum values on $\Cl{W}$. Let
\[
M = \max_{\mathbf{x} \in \Cl{W}} F(\mathbf{x})\,, \quad
m = \min_{\mathbf{x} \in \Cl{W}} F(\mathbf{x})\,.
\]

If $m = M$, then $F$ is constant on $\Cl{W}$, which is impossible, since $H^{n-1} \subset F^{-1}(a)$ and $\mathbf{y} \notin F^{-1}(a)$. Therefore, either $m \neq a$, or $M \neq a$.

Let us suppose that $m \neq a$.

Obviously, $\emptyset \neq F^{-1}(m) \cap \Cl{W} \subset W$ and $(F^{-1}(m) \cap \Cl{W}) \subset \Sigma(F({\bf x}))$ since every point of this set is a local minimum of $F$.
Therefore, if $F(\mathbf{y}) \neq m$ then $U$ contains other critical points of $F$ distinct from $\mathbf{y}$.
If $F(\mathbf{y}) = m$, then the set $F^{-1}(m) \cap \Cl{W} = F^{-1}(F(\mathbf{y})) \cap W$ is the subset of $\Sigma(F({\bf x}))$ and also contains more than one point because $\mathbf{y}$ is not an isolated point of its level set.
This contradicts to our initial assumption that $\mathbf{y}$ is an isolated critical point of $F$.

The case $M \neq a$ is considered similarly.

From the arbitrariness in the choice of $\varepsilon > 0$ we conclude that $\mathbf{y}$ is not an isolated point of the set $\Sigma(F({\bf x}))$.
$\square$

\begin{rem}
With the help of technique from~\cite{Dancer} it can be proved that Corollary~\ref{corr:non-isolated_crit_point} is valid for $F \in C^r(G)$, $r \geq 2$. But this is outside the scope of the current discussion.
\end{rem}

\medskip

Now we shall derive some consequences from Theorem~\ref{th_surf_family}

Let $G$ be a domain in $\bR^n$ and let $F \in C^2(G)$. Consider the \emph{gradient system} of $F$ on $G$
\[
\frac{dx}{dt} = - \grad F(x)
\]
where
\[
\grad F(x) = \left( \frac{\partial F}{\partial x_1}, \ldots, \frac{\partial F}{\partial x_n} \right)^T \,.
\]

\begin{theorem}\label{theorem_gradient_system}
Let $G$ be a domain in $\bR^n$. Suppose $F \in C^n(G)$ and $x_0 \in G$ be a connected component of the level set $F^{-1}(F(x_0))$.

Then either the gradient system of $F$ or of $-F$ on $G$ is Lyapunov stable in $x_0$.
\end{theorem}

\begin{proof}
It follows from Theorem~\ref{th_surf_family} that we can select a sequence $\{a_i\}_{i \in \bN}$ of regular values of $F$ that converges to $F(x_0)$ and a sequence $\{H^{n-1}_i\}_{i \in \bN}$ of nested connected hypersurfaces that converge to $x_0$ and such that each $H^{n-1}_i$ is a connected component of $F^{-1}(F(a_i))$.

Let $\vec{N_i}({\bf x})$, ${\bf x} \in H^{n-1}_i$, be a unique normal vector field of unit length on $H^{n-1}_i$ such that it directs towards the internal component of the complement $\bR^n \setminus H^{n-1}_i$. It is obvious that vectors $\grad F({\bf x})$ and $\vec{N_i}({\bf x})$ are collinear for all ${\bf x} \in H^{n-1}_i$, $i \in \bN$.

Since each $a_i$ is regular value of $F$, then $\grad F({\bf x}) \neq 0$ for all ${\bf x} \in H^{n-1}_i$, $i \in \bN$.
Therefore, $S({\bf x}) = <-\grad F({\bf x}), \vec{N_i}({\bf x})> \neq 0$ for all ${\bf x} \in H^{n-1}_i$, $i \in \bN$.
Function $S({\bf x})$ is nonzero and continuous on each connected set $H^{n-1}_i$, consequentely it is sign-definite on each $H^{n-1}_i$.

Let us consider two subsequences of $\{H^{n-1}_i\}_{i \in \bN}$
\[
\begin{aligned}
S_{+} & = \{H^{n-1}_i \;|\; S({\bf x}) > 0 \textrm{ on } H^{n-1}_i \}\,, \\
S_{-} & = \{H^{n-1}_i \;|\; S({\bf x}) < 0 \textrm{ on } H^{n-1}_i \}\,.
\end{aligned}
\]
At least one of these subsequences contains an infinite number of elements.

If $|S_{+}| = \infty$, then we can take $S_{+}$ and apply Theorem~\ref{th:lyapunov_from_ge0} to the system $\frac{d {\bf x}}{dt} = - \grad F({\bf x})$.

If $|S_{-}| = \infty$, then we take $S_{-}$ and observe that $<-\grad (-F({\bf x})), \vec{N_i}({\bf x})> = -S({\bf x}) > 0$ on every element $H^{n-1}_i$ of $S_{-}$. Therefore, applying Theorem~\ref{th:lyapunov_from_ge0} to the system $\frac{d {\bf x}}{dt} = - \grad (-F({\bf x}))$
we conclude that this system is Lyapunov stable in $x_0$.
\end{proof}

Let $G$ be a domain in $\bR^{2n}$ and let $F \in C^2(G)$ where $F = F(y, z)$ with $y, z \in \bR^n$. Consider the \emph{Hamiltonian system} with $n$ degrees of freedom on $G$
\[
\begin{aligned}
\frac{dy}{dt} & = \frac{\partial F}{\partial z} \\
\frac{dz}{dt} & = -\frac{\partial F}{\partial y} \,,
\end{aligned}
\]
where
\[
\begin{aligned}
\frac{\partial F}{\partial y} & = \left( \frac{\partial F}{\partial y_1}, \ldots, \frac{\partial F}{\partial y_n} \right)^T \\
\frac{\partial F}{\partial z} & = \left( \frac{\partial F}{\partial z_1}, \ldots, \frac{\partial F}{\partial z_n} \right)^T \,.
\end{aligned}
\]

\begin{theorem}\label{theorem_Hamiltonian_system}
Consider a domain $G \in \bR^{2n}$.
Let $F \in C^{2n}(G)$ and $x_0 \in G$ be a connected component of the level set $F^{-1}(F(x_0))$.

Then the corresponding Hamiltonian system on $G$ is Lyapunov stable in $x_0$.
\end{theorem}

\begin{proof}
Let us denote
\[
{\bf x} = (y, z) \,, \quad \myvectorfield({\bf x}) = \left( \left(\frac{\partial F}{\partial z}\right)^T, \left(-\frac{\partial F}{\partial y}\right)^T \right)^T \,.
\]
In this notation our system has the form~\eqref{eq_system_1}.

We make use of Theorem~\ref{th_surf_family} and select a sequence $\{H^{n-1}_i\}_{i \in \bN}$ of nested connected hypersurfaces
that converges to $x_0$ and such that each $H^{n-1}_i$ is contained in a level set of $F$. Let $\vec{N_i}({\bf x})$, ${\bf x} \in H^{n-1}_i$, be a normal vector field of unit length on $H^{n-1}_i$ such that it directs towards the internal component of the complement $\bR^{2n} \setminus H^{n-1}_i$.

It is known that the trajectories of Hamiltonian system lie on the level surfaces of $F$. Therefore $S_i({\bf x}) = \;<\vec{N_i}({\bf x}), \myvectorfield({\bf x})>\; = 0$ for every ${\bf x} \in H^{n-1}_i$, $i \in \bN$, and we are in the conditions of Theorem~\ref{th:lyapunov_from_ge0}.
Applying it we conclude that $x_0$ is stable in the sense of Lyapunov for our Hamiltonian system.
\end{proof}

\begin{rem}
So, it turns out that in order to check Lyapunov stability of gradient or Hamiltonian systems at a critical point $x_0 \in G$ of a function $F \in C^n(G)$ it suffices to verify that this point is the connected component of its level set $F^{-1}(F(x_0))$.
\end{rem}

%

\section{Proof of Theorem~\ref{th_surf_family}.}
\label{sec:th_surf_family_proof}

Before we proceed with the proof of theorem, let us consider some necessary auxiliary statements.

\subsection{On closed hypersurfaces in $\bR^{n}$.}

\begin{defn}[see.~\cite{Kuratowski}]
Let $X$ be a metric space and $A \subset X$. We denote by $LC(A)$ a set of all $x \in \Cl{A}$ with the following property: there exists an open neighbourhood $G$ of $x$ of an arbitrary small diameter such that $G \cap A$ is connected.
\end{defn}

Let $\rho$ be a metrics in $\bR^{n}$. We designate by $U_{\varepsilon}(A)$ the $\varepsilon$-neighbourhood of a set $A \subset \bR^{n}$:
\[
U_{\varepsilon}(A) = \{ x \in \bR^n \,|\, \inf_{y \in A}(\rho(x, y) < \varepsilon \} \,.
\]

\begin{lemma}\label{lemma_connected_collar}
Let $W$ be a domain in $\bR^{n}$. Suppose the frontier $R = \Fr(W)$ is connected and $R \subset LC(W)$.

Then $W_{\varepsilon} = W \cap U_{\varepsilon}(R)$ is connected for every $\varepsilon > 0$.
\end{lemma}

\begin{proof} Fix $\varepsilon > 0$.

Let $x_1$, $x_2 \in W_{\varepsilon}$. Let $x^0_1$ and $x^0_2$ be the closest points of $R$ to $x_1$ and $x_2$ accordingly. Let also $\gamma_1, \gamma_2 : \bR \to \bR^n$ be continuous curves which comply with the correlations
\begin{itemize}
	\item $\gamma_1(0) = x_1$, $\gamma_1(1) = x^0_1$, $\gamma_2(0) = x_2^0$, $\gamma_2(1) = x_2$;
	\item $\gamma_1[0, 1) \cup \gamma_2(0, 1] \subset W_{\varepsilon}$.
\end{itemize}
We can take for instance $\gamma_1(t) = (1-t) x_1 + t x_1^0$, $\gamma_2(t) = (1-t) x_2^0 + t x_2$, $t \in I$.

We use the inclusion $R \subset LC(W)$ from condition of lemma and choose for every $x \in R$ an open neighbourhood $G(x)$ in $\bR^n$ which is contained in $U_{\varepsilon}(x)$ and such that the set $V(x) = G(x) \cap W$ is connected. Then $V(x) \subset W_{\varepsilon}$, $x \in R$.

It is known (see~\cite{Kuratowski}) that for an open cover of a connected space we can connect every pair of points of this space by a finite chain which consists of elements of this cover. Thus, with a pair of points $x_1^0$, $x_2^0 \in R$ one can associate a finite set of points $y_1, \ldots, y_s$ such that $x_1^0 \in G(y_1)$, $x_2^0 \in G(y_s)$ and
\[
G(y_i) \cap G(y_{i+1}) \cap R \neq \emptyset \text{ for all } i \in \{1, \ldots, s-1\} \,.
\]

Since $x_1^0 \in G(y_1)$, we have that $\gamma_1[0, 1) \cap G(y_1) \neq \emptyset$. And from $\gamma_1[0, 1) \subset W_{\varepsilon}$ it follows that $\gamma_1[0, 1) \cap V(y_1) \neq \emptyset$. Similarly, $\gamma_2(0, 1] \cap V(y_s) \neq \emptyset$.

Let $i \in \{1, \ldots, s-1\}$. The nonempty set $G(y_i) \cap G(y_{i+1}) \cap R$ is contained in $\Fr W$. Consequently, its neighbourhood $G(y_i) \cap G(y_{i+1})$ in $\bR^n$ intersects $W$ and $V(y_i) \cap V(y_{i+1}) \neq \emptyset$.

Thus, all members of the union
\[
R = \gamma_1[0, 1) \cup V(y_1) \cup \ldots \cup V(y_s) \cup \gamma_2(0, 1]
\]
are connected and each pair of adjacent sets in this sequence have a common point. Therefore, $R$ is connected set. Furthermore, by construction it lies in $W_{\varepsilon}$ and contains points $x_1 = \gamma_1(0)$ and $x_2 = \gamma_2(1)$.

From the arbitrariness of a choice of $x_1$, $x_2 \in W_{\varepsilon}$ it follows that the set $W_{\varepsilon}$ is connected.
\end{proof}

\begin{corr}\label{corr_conectedness}
Let $N$ be a connected closed hypersurface in $\bR^n$. Let $W$ be a connected component of the complement $\bR^n \setminus N$.

Then the intersection $W_{\varepsilon} = W \cap U_{\varepsilon}(N)$ is connected for every $\varepsilon > 0$.
\end{corr}

\begin{proof}
A closed hypersurface in $\bR^n$ splits $\bR^n$ (see~\cite{two_components}), therefore $\bR^n \setminus (W \cup N) \neq \emptyset$.

Let us consider following functions.
\begin{eqnarray*}
\chi(x) & = &
	\left\{
	\begin{aligned}
		1, &\quad\text{when } x \in W \,, \\
		-1, &\quad\text{otherwise} \,,
	\end{aligned}
	\right.
\\
\Phi(x) & = & \rho(x, N) \,,\\
\Psi(x) & = & \chi(x) \cdot \Phi(x) \,.
\end{eqnarray*}
Obviously, $\Phi$ is continuous in $\bR^n$ and both $\chi$ and $\Psi$ are continuous at all points of the open set $(\bR^n \setminus N) \subseteq (\bR^n \setminus \Fr W)$.

It is clear that $N = \Psi^{-1}(0)$. Furthermore, $\Phi(x) = |\Psi(x)|$, $x \in \bR^n$. So, $\Psi^{-1}(-\varepsilon, \varepsilon) = \Phi^{-1}(-\varepsilon, \varepsilon)$ is open for every $\varepsilon > 0$. Hence $\Psi$ is also continuous at every $x \in N$.

Let us examine two subsets of $N$.
\begin{multline*}
N_1 = \{ x \in N \,|\, \exists\, \varepsilon > 0 : \Psi(y) \leq 0 \quad \forall\, y \in U_{\varepsilon}(x) \} \, \cup \\
	\cup \{ x \in N \,|\, \exists\, \varepsilon > 0 : \Psi(y) \geq 0 \quad \forall\, y \in U_{\varepsilon}(x) \} \,,
\end{multline*}
\[
N_2 = \{ x \in N \,|\, \forall\, \varepsilon > 0 \;\exists\, y_1, y_2 \in U_{\varepsilon}(x) : \Psi(y_1) \Psi(y_2) < 0 \} \,.
\]

Relations $N_1 \cap N_2 = \emptyset$ and $N_1 \cup N_2 = N$ are obviously fulfilled.

By definition of hypersurface for every point $x \in \bR^n$ there are a neighbourhood $V_x$ in $\bR^n$ and a diffeomorphism $\psi_x : V_x \to \bR^n$, such that $\psi_x(V) = \bR^n$, $\psi_x(V \cap N) = \bR^{n-1} \times \{0\}$.

By construction the sign of $\Psi$ is fixed on each connected component of the complement $\bR^n \setminus N$. Therefore every $x \in N$ is contained in one of the sets $N_1$ or $N_2$ together with its neighbourhood $V_x \cap N$, and both $N_1$ and $N_2$ are open in $N$. Moreover, if $x \in N_2$ then exactly one of two components of $V_x \setminus N$ belongs to $W$. Consequently, $N_2 \subset LC(W)$.

According to the condition of this corollary $N$ is connected. Hence either $N_1 = \emptyset$ or $N_2 = \emptyset$.

It is clear that $\Cl{W} \subseteq (W \cup N)$. Moreover $\Fr \Cl{W} \neq \emptyset$ since $\bR^n \setminus (W \cup N) \neq \emptyset$. It is straightforward that $\Fr \Cl{W} = N_2$. So $N_2 \neq \emptyset$, hence $N_1 = \emptyset$ and $N = N_2 \subset LC(W)$. Finally since $\Fr \Cl{W} \subseteq \Fr W \subseteq N$, it follows that $N = \Fr W$ and we can apply lemma~\ref{lemma_connected_collar}.
\end{proof}

\begin{lemma}\label{lemma_complements_intersection}
Let $N_1$ and $N_2$ be closed hypersurfaces in $\bR^n$ such that $N_1 \cap N_2 = \emptyset$.

Let $V_1$ and $V_2$ be connected components of the sets $\bR^n \setminus N_1$ and $\bR^n \setminus N_2$, respectively.

Then $V_1 \cap V_2$ is connected.
\end{lemma}

\begin{proof}
Suppose $x_1$, $x_2 \in V_1 \cap V_2$. Let us verify that the set $V_1 \cap V_2$ contains a connected subset which includes $x_1$ and $x_2$.

It is known that open connected subsets of $\bR^n$ are arcwise connected. So, $V_1$ and $V_2$ are arcwise connected sets.

Let $\gamma : I \to \bR^n$ be a continuous path which connects $x_1$ to $x_2$ in $V_1$, i. e. $\gamma(0) = x_1$, $\gamma(1) = x_2$ and $\gamma(I) \subset V_1$.

Since $N_2$ is a closed hypersurface (i. e. compact and borderless), it has a finite number of connected components. Let us designate them by $N_2^{1}, \ldots, N_2^{m}$.

Denote $\tau_1' = \inf \{ t \in I \,|\, \gamma(t) \in N_2 \}$. Since $N_2$ is closed in $\bR^n$, we have that $\gamma(\tau_1') \in N_2$. Moreover $\tau_1' > 0$, as $x_1 = \gamma(0) \in V_1 \cap V_2$. Let $\gamma(\tau_1') \in N_2^{\sigma(1)}$. Write $\tau_1'' = \sup \{ t \in I \,|\, \gamma(t) \in N_2^{\sigma(1)} \}$. Then
\begin{itemize}
	\item $\gamma(\tau_1'') \in N_2^{\sigma(1)}$,
	\item $\tau_1'' < 1$, since $x_2 = \gamma(1) \in V_1 \cap V_2$,
	\item $\gamma(t) \notin N_2^{\sigma(1)}$ for all $t > \tau_1''$.
\end{itemize}

If $\gamma(\tau_1'', 1] \cap N_2 \neq \emptyset$, there exists $\tau_2' = \inf \{ t > \tau_1'' \,|\, \gamma(t) \in N_2 \}$. Observe that $\tau_2' > \tau_1''$. Indeed, the compacts $N_2^{\sigma(1)}$ and $N \setminus N_2^{\sigma(1)}$ have disjoint neighbourhoods, therefore there exists $\varepsilon > 0$, such that $\gamma(t) \in N_2^{\sigma(1)}$ as soon as correlations $t \in \gamma^{-1}(N_2)$ and $|t - \tau_1''| < \varepsilon$ are fulfilled.

As above it is verified that $\gamma(\tau_2') \in N_2^{\sigma(2)}$ for a certain $\sigma(2) \in \{1, \ldots, m\}$. By definition we have $\gamma(\tau_1'', \tau_2') \cap N_2 = \emptyset$.

Denote $\tau_2'' = \sup \{ t \in I \,|\, \gamma(t) \in N_2^{\sigma(2)} \}$. Then
\begin{itemize}
	\item $\gamma(\tau_2'') \in N_2^{\sigma(2)}$,
	\item $\tau_2'' < 1$,
	\item $\gamma(t) \notin N_2^{\sigma(1)} \cup N_2^{\sigma(2)}$ for every $t > \tau_2''$.
\end{itemize}

Suppose that we have already constructed numbers
\begin{equation}\label{eq_parameters}
0 < \tau_1' \leq \tau_1'' < \tau_2' \leq \tau_2'' < \cdots < \tau_k' \leq \tau_k'' < 1 \,,
\end{equation}
such that
\begin{itemize}
	\item[(a)] $\bigl( \gamma[0, \tau_1') \cup \gamma(\tau_1'', \tau_{2}') \cup \cdots \cup \gamma(\tau_{k-1}'', \tau_{k}') \bigr) \cap N_2 = \emptyset$;
	\item[(b)] $\gamma(\tau_i')$, $\gamma(\tau_i'') \in N_2^{\sigma(i)}$, $i \in \{1, \ldots, k\}$;
	\item[(c)] $\gamma(t) \notin N_2^{\sigma(1)} \cup \cdots \cup N_2^{\sigma(i)}$ when $t > \tau_i''$, $i \in \{1, \ldots, k\}$.
\end{itemize}
Let $\gamma(\tau_k'', 1] \cap N_2 \neq \emptyset$. Designate $\tau_{k+1}' = \inf \{ t \in (\tau_k'', 1] \,|\, \gamma(t) \in N_2 \}$.

Since $\gamma(\tau_k'') \in N_2^{\sigma(k)}$ and compacts
\[
\bigcup_{i = 1}^k N_2^{\sigma(i)} \quad \text{and} \quad N_2 \setminus \Bigl( \bigcup_{i = 1}^k N_2^{\sigma(i)} \Bigr)
\]
do not intersect, we obtain that $\tau_{k+1}' > \tau_k''$. It is also clear that $\gamma(\tau_k'', \tau_{k+1}') \cap N_2 = \emptyset$.

As the set $N_2$ is compact, we get that $\gamma(\tau_{k+1}') \in N_2$. So there exists $\sigma(k+1)$, such that $\gamma(\tau_{k+1}') \in N_2^{\sigma(k+1)}$. Denote $\tau_{k+1}'' = \sup \{ t \in I \,|\, \gamma(t) \in N_2^{\sigma(k+1)} \}$. Then
\begin{itemize}
	\item $\gamma(\tau_{k+1}'') \in N_2^{\sigma(k+1)}$;
	\item $\tau_{k+1}'' < 1$;
	\item $\gamma(t) \notin N_2^{\sigma(1)} \cup \cdots \cup N_2^{\sigma(i)}$ whenever $t > \tau_i''$, $i \in \{1, \ldots, k+1\}$.
\end{itemize}
Consequently, the sequence
\[
0 < \tau_1' \leq \tau_1'' < \tau_2' \leq \tau_2'' < \cdots < \tau_{k+1}' \leq \tau_{k+1}'' < 1
\]
complies with properties which are similar to (a)--(c).

Observe that it follows from (b) and (c) that all numbers $\sigma(i)$ are distinct, $\sigma(i) \in \{1, \ldots, m\}$, $i \in \{1, \ldots, k\}$. Therefore, if the sequence~\eqref{eq_parameters} complies with the properties (a)--(c), then $k \leq m$. Consequently, there exists $k \leq m$, such that if the sequence~\eqref{eq_parameters} satisfies to (a)--(c), then it also meets the following property
\[
\gamma(\tau_k'', 1] \cap N_2 = \emptyset \,.
\]

As a matter of convenience we reindex connected components of $N_2$ in order to satisfy equalities $\sigma(i) = i$, $i \in \{1, \ldots, k\}$. Then the sequence~\eqref{eq_parameters} meets the following properties:
\begin{itemize}
	\item[(a$'$)] $\bigl( \gamma[0, \tau_1') \cup \gamma(\tau_1'', \tau_{2}') \cup \cdots \cup \gamma(\tau_{k-1}'', \tau_{k}') \cup \gamma(\tau_k'', 1] \bigr) \cap N_2 = \emptyset$;
	\item[(b$'$)] $\gamma(\tau_i')$, $\gamma(\tau_i'') \in N_2^i$, $i \in \{1, \ldots, k\}$;
	\item[(c$'$)] $\gamma(t) \notin N_2^{1} \cup \cdots \cup N_2^{i}$ when $t > \tau_i''$, $i \in \{1, \ldots, k\}$.
\end{itemize}

Notice that all the sets $N_1$, $N_2^1, \ldots, N_2^m$ are disjoint compacts. So there exists an $\varepsilon > 0$ complying with the following equalities:
\begin{align*}
U_{\varepsilon}(N_1) \cap U_{\varepsilon}(N_2^i) & = \emptyset\,, \quad i \in \{1, \ldots, m\} \,; \\
U_{\varepsilon}(N_2^i) \cap U_{\varepsilon}(N_2^j) & = \emptyset\,, \quad i \neq j\,,\quad i, j \in \{1, \ldots, m\} \,.
\end{align*}

Let $W_i$, $i \in \{1, \ldots, m\}$, be a connected component of $\bR^n \setminus N_2^i$, such that $W_i \cap V_2 \neq \emptyset$. It is easy to see that $V_2 \subseteq W_i$ for every $i$. Thus, the component $W_i$ is uniquely determined for each $i$, and moreover $x_1, x_2 \in W_i$, $i \in \{1, \ldots, m\}$.

Observe that the sets
\begin{align*}
K_i' & = \gamma[0, \tau_1') \cup \Bigl( \bigcup_{j < i} N_2^j \Bigr) \cup \Bigl( \bigcup_{2 \leq j \leq i} \gamma(\tau_{j-1}'', \tau_j') \Bigr) \,, \\
K_i'' & = \gamma(\tau_i'', 1]
\end{align*}
are connected for all $i \in \{1, \ldots, m\}$, since all sets $N_2^j$ are so and conditions (b$'$) are fulfilled.

It is also true that $K_i' \cup K_i'' \subset W_i$, $i \in \{1, \ldots, k\}$. In fact, on one hand it follows from (a$'$) that $K_i' \cap N_2^i = \emptyset$, and (c$'$) implies $K_i'' \cap N_2^i = \emptyset$; on the other hand, $x_1 = \gamma(0) \in K_i' \cap W_i$ and $x_2 = \gamma(1) \in K_i'' \cap W_i$.

It follows from what has been said that
\[
\gamma(\tau_{i-1}'', \tau_i') \cup \gamma(\tau_i'', \tau_{i+1}') \subset W_i \,, \quad i \in \{1, \ldots, k\} \,.
\]
Together with the condition (b$'$) this results in the inequalities
\begin{equation}\label{eq_intervals}
\gamma(\tau_{i-1}'', \tau_i') \cap W_{i, \varepsilon} \neq \emptyset \,, \quad
\gamma(\tau_{i}'', \tau_{i+1}') \cap W_{i, \varepsilon} \neq \emptyset \,, \quad
i \in \{1, \ldots, k\} \,.
\end{equation}
We designated here $\tau_0'' = 0$, $\tau_{k+1}' = 1$, $W_{i, \varepsilon} = W_i \cap U_{\varepsilon}(N_2^i)$, $i \in \{1, \ldots, k\}$.

Now Corollary~\ref{corr_conectedness} and correlations~\eqref{eq_intervals} imply that the set
\[
K = \gamma[0, \tau_1') \cup W_{1, \varepsilon} \cup \gamma(\tau_1'', \tau_2') \cup \ldots \cup \gamma(\tau_{k-i}'', \tau_k') \cup W_{k, \varepsilon} \cup \gamma(\tau_k'', 1]
\]
is connected. Indeed, all sets in this union are connected, and from~\eqref{eq_intervals} it follows that every two adjacent sets in this chain have a common point.

In addition, by virtue of choice of $\varepsilon > 0$ we have
\[
W_{i, \varepsilon} \cap (N_1 \cup N_2) = \emptyset \,, \quad i \in \{1, \ldots, m\} \,.
\]
Therefore, it follows from the choice of $\gamma$ and from conditions (a$'$) that $K \cap (N_1 \cup N_2) = \emptyset$.

So, we have constructed a connected set $K$, which contains points $x_1 = \gamma(0)$ and $x_2 = \gamma(1)$ and does not intersect surfaces $N_1$ and $N_2$.

>From the arbitrariness in the choice of points $x_1, x_2 \in V_1 \cap V_2$ we conclude that the set $V_1 \cap V_2$ is connected.
\end{proof}

\begin{corr}\label{corr_intersection}
Under the condition of Lemma~\ref{lemma_complements_intersection} the set $V_1 \cap V_2$ is the connected component of the complement $\bR^n \setminus (N_1 \cup N_2)$.
\end{corr}

\begin{proof}
By Lemma~\ref{lemma_complements_intersection} the set $V_1 \cap V_2$ is connected. Moreover, it is easy to see from condition of Lemma~\ref{lemma_complements_intersection} that this set does not intersect $N_1 \cup N_2$. Therefore, there is a component $W$ of the complement $\bR^n \setminus (N_1 \cup N_2)$ which contains $V_1 \cap V_2$.

Suppose that Corollary is invalid. Then $W \neq (V_1 \cap V_2)$ and there exists $x \in W \cap ((\bR^n \setminus V_1) \cup (\bR^n \setminus V_2))$.

Let $x \in (\bR^n \setminus V_1)$. Since $W \cap V_1 \supset W \cap (V_1 \cap V_2) = V_1 \cap V_2 \neq \emptyset$, the set $W \cup V_1$ is connected. By construction $W \cap N_1 = \emptyset$, so $W \cup V_1 \subset (\bR^n \setminus N_1)$. In this case $(W \cup V_1) \supsetneq V_1$, as $x \in W \setminus V_1$ by our hypothesis. We obtain a contradiction to the condition of Lemma~\ref{lemma_complements_intersection} which says that $V_1$ is the connected component of the complement $\bR^n \setminus N_1$. Consequently, our supposition is false and $x \in V_1$.

The inclusion $x \in V_2$ is proved similarly.

Therefore, $W = V_1 \cap V_2$. Corollary is proved.
\end{proof}

\begin{corr}\label{corr_component}
Let $x_0 \in \bR^n$. Let $N$ be a closed hypersurface in $\bR^n$ and $W$ be the component of the complement $\bR^n \setminus N$, such that $x_0 \in W$. Suppose the set $W$ is bounded.

Then there exists a connected component $N_0$ of $N$, such that the connected component $W_0$ of $\bR^n \setminus N_0$ which contains $x_0$ is bounded.
\end{corr}

\begin{proof}
Since $N$ is compact, and so bounded in $\bR^n$, we can assume without loss of generality that $N$ is contained in the unit ball $B$ which has its center at the origin.

Denote by $S$ the unit sphere $\Fr B$. It is clear that $S \subset \bR^n \setminus N$. Boundedness of $W$ means that $x_0$ and the connected set $S$  belong to distinct components of the complement $\bR^n \setminus N$.

Let $N_1, \ldots, N_m$ be the connected components of $N$. Let $W_i$ be the component of the complement $\bR^n \setminus N_i$, such that $x_0 \in W_i$, $i \in \{1, \ldots, m\}$.

Suppose that the conclusion of the Corollary is false. This is equivalent to the claim that $S \subset W_i$ for all $i \in \{1, \ldots, m\}$. By the sequential application of Lemma~\ref{lemma_complements_intersection} and of Corollary~\ref{corr_intersection} to the pairs of sets
\[
N'(i) = \bigcup_{j=1}^i N_j \,, \quad N''(i) = N_{i+1} \,, \quad i \in \{1, \ldots, m-1\} \,,
\]
we verify that the set $W_1 \cap \cdots \cap W_m \subset \bR^n \setminus N$ is connected. Therefore, both $x_0$ and $S$ are contained in the same connected component of the complement $\bR^n \setminus N$. But this contradicts to the condition of Corollary.
\end{proof}

\subsection{Proof of Theorem~\ref{th_surf_family}}

Without loss of generality we can assume that $F(x_0) = 0$.

We shall use the following designations throughout the proof:
\[
B_{\delta} = \{ y \in \bR^n \,|\, \rho(y, x_0) \leq \delta \}\,.
\]
\[
\mathring{B}_{\delta} = \Int(B_{\delta}) = \{ y \in \bR^n \,|\, \rho(y, x_0) < \delta \}\,.
\]

\medskip

{\bf Necessity.}
Let $F$ be an $\bf {L}$-function for the point $x_0$.

Let us select a sequence $\{a_i\}_{i \in \bN}$ of regular values of $F$ and a sequence $\{ H_{i}^{n-1} \subset  F^{-1}(a_i)\}_{i \in \bN}$ of connected components of level sets of $F$ such that they comply with Definition~\ref{defn_L_function}.

Denote by $C$ the connected component of $F^{-1}(0)$ which contains $x_0$.

Suppose that $C \neq \{x_0\}$ contrary to the statement of Theorem. Then there exists $x_1 \in C$, $x_1 \neq x_0$. Denote $\delta = \rho(x_0, x_1)/2$.

Since the sequence $\{H_{i}^{n-1}\}$ converges to $x_0$, there is an index $M \in \bN$ such that $H_M^{n-1} \subset B_\delta$.

On one hand, by definition the point $x_0$ is contained in the bounded component of the complement $\bR^n \setminus H_M^{n-1}$. However, it is straightforward that $x_1$ is contained in the unbounded component of this complement by the choice of $H_M^{n-1}$. Consequently, the connected set $C$ must intersect $H_M^{n-1}$.

It easily follows from the relation $H_M^{n-1} \cap C \neq \emptyset$ that $H_M^{n-1} \cup C \subset F^{-1}(0)$. The set $H_M^{n-1} \cup C$ is connected and does not coincide with $H_M^{n-1}$ since $\{x_0, x_1\} \subset C \setminus H_M^{n-1}$. This contradicts to the choice of $H_M^{n-1}$ being the connected component of the corresponding level set of $F$.

The contradiction obtained proves that $C = \{x_0\}$ and $x_0$ is the quasi-isolated point of $F$.

\medskip

{\bf Sufficiency.}
Let $x_0$ be a quasi-isolated point of $F$.

It follows from Sard's  Theorem~\cite{Sard} that the set of regular values of $F$ is residual and everywhere dense. So there exists a decreasing sequence of positive real numbers $\{\varepsilon_i\}_{i \in \bN}$, which complies with the following properties:
\begin{itemize}
	\item $\lim_{i \to \infty} \varepsilon_i = 0$;
	\item all numbers $\pm\varepsilon_i$ are regular values of $F$.
\end{itemize}

Denote by $U^i$ the connected component of the open set $Q^i = \{ x \in G \,|\, -\varepsilon_i < F(x) < \varepsilon_i \}$, such that $x_0 \in U^i$. Obviously, $U^j \subset U^i$ when $i < j$.

We fix $\delta > 0$ small enough to satisfy the inclusion $B_\delta \subset G$.
Let us denote by $U^i_{\delta}$ the component of $U^i \cap B_{\delta}$, such that $x_0 \in U^i_{\delta}$. Let also
\[
K^i_{\delta} = \Cl{U^i_{\delta}} \,, \quad i \in \bN \,.
\]
It obviously follows from the continuity of $F$ that
\[
K^i_{\delta} \subset \{ x \in G \,|\, -\varepsilon_i \leq F(x) \leq \varepsilon_i \}\,.
\]
We have also $U^j_{\delta} \subset U^i_{\delta}$ and $K^j_{\delta} \subset K^i_{\delta}$ for $i < j$ by construction.

Let us consider the set
\[
K_{\delta} = \bigcap_{i \in \bN} K^i_{\delta} \,.
\]

On one hand $x_0 \in K_{\delta} \subseteq F^{-1}(0)$, since $\varepsilon_i \to 0$ when $i \to \infty$.

On the other hand, $\{K^i_{\delta}\}$ is the sequence of embedded connected compacts. Hence, $K_{\delta}$ is connected.

Thus, from the condition of Theorem we conclude that $K_{\delta} = \{x_0\}$.

Then there exists an $m \in \bN$, such that $U^r_{\delta} \subset K^r_{\delta} \subset \mathring{B}_{\delta}$ for all $r \geq m$.

Recall that by construction the set $U^i$ is connected component of the open subset $Q^i$ of $G$.
The space $\bR^n$ is locally connected, therefore $U^i$ is open in $\bR^n$ (see~\cite{Kuratowski}). Consequently,  $U^i_{\delta}$ is open in $B_{\delta}$. In fact, the space $B_{\delta}$ is locally connected, as it is the homeomorphic image of closed disk. The set $U^i_{\delta}$ is a connected component of $U^i \cap B_{\delta}$, so $U^i \cap B_{\delta}$ is open in $B_{\delta}$.

In this way, if the inclusion $U^i_{\delta} \subset \mathring{B}_{\delta} = \Int(B_{\delta})$ is valid, then the set $U^i_{\delta}$ is open in $\bR^n$.

On the other hand, it is known (see~\cite{Kuratowski}), that if arbitrary sets $A$ and $B$ satisfy the inclusion $A \subset B \subset \Cl{A}$, then connectedness of $A$ results in the connectedness of $B$. So, the set $K^i_{\delta} \cap U^i$ is connected. Indeed, the set $U^i_{\delta}$ is connected and $U^i_{\delta} \subset K^i_{\delta} \cap U^i \subset \Cl{U^i_{\delta}} = K^i_{\delta}$.

Consequently, $K^i_{\delta} \cap U^i = U^i_{\delta}$, since $K^i_{\delta} \cap U^i \subseteq K^i_{\delta} \cap (U^i \cap B_{\delta})$ and $U^i_{\delta}$ is a connected component of $U^i \cap B_{\delta}$.

Thus, the compact $K^i_{\delta}$ complies with the correlation $U_i \setminus U^i_{\delta} = U_i \setminus K^i_{\delta}$, so the set $U_i \setminus U^i_{\delta}$ is open in $\bR^n$.

From what has been said we conclude that $U^r_{\delta} = U^r$ for $r \geq m$. Indeed, $U^r = U^r_{\delta} \sqcup (U^r \setminus U^r_{\delta})$, moreover in our case both the sets $U^r_{\delta} \ni x_0$ and $(U^r \setminus U^r_{\delta})$ are open. It now follows from the connectedness of $U^r$ that $(U^r \setminus K^r_{\delta}) = \emptyset$.

Let $r \geq m$.

Consider the set $N_r = \Fr U^r$.
The set $U^r$ is open, so $N_r \cap U^r = \emptyset$. On the other hand, $U^r$ is the connected component of $Q^r = \{ x \in M^n \,|\, -\varepsilon_r < F(x) < \varepsilon_r \}$, therefore $\Cl{U^r} \cap Q^r = K^r_{\delta} \cap Q^r = U^r$. It implies that $N_r \subseteq F^{-1}(-\varepsilon_r) \cup F^{-1}(\varepsilon_r)$.

Both $-\varepsilon_r$ and $\varepsilon_r$ are regular values of $F$. Hence by the implicit function theorem (see~\cite{PalisDiMelo})
the space $N_r$ is locally diffeomorphic to $\bR^{n-1}$ at every point.
Therefore, $N_r$ is the closed hypersurface in $\bR^n$.

Applying Corollary~\ref{corr_component} to $N_r$, we conclude that there exists a component $N_r^0$ of $N_r$, which complies with the following property: if a component $W_r$ of the complement $\bR^n \setminus N_r^0$ contains $x_0$, then $W_r$ is limited.

We have proved already that $N_r \subset K^r_{\delta} \subset \mathring{B}_{\delta}$, hence $W_r \subset \mathring{B}_{\delta}$ and $\Cl{W_r} \subset B_{\delta}$.

Let $B_{\delta_1} = B_{\delta}$. We already found $m_1 = m \in \bN$, such that there exists a component $N_{m_1}^0$ of the hypersurface $N_{m_1}$, which has the following property: if a component $W_{m_1}$ of the complement $\bR^n \setminus N_{m_1}^0$ contains $x_0$, then its closure $K_{m_1} = \Cl{W}_{m_1}$ belongs to $B_{\delta_1}$.

Let us assume now that for some $k \in \bN$ the following objects are already determined:
\begin{itemize}
	\item a sequence of positive numbers $\delta_1, \ldots, \delta_k$, such that $\delta_{i+1} < \delta_i/2$, $i \in \{1, \ldots, k-1\}$;
	\item a family $B_{\delta_1}, \ldots, B_{\delta_k}$ of neighbourhoods of $x_0$, $B_{\delta_i} = \{ y \in \bR^n \,|\, \rho(y, x_0) \leq \delta_i \} \subset G$, $i \in \{1, \ldots, k\}$;
	\item a sequence of numbers $m_1 < \cdots < m_k$;
	\item a collection of connected components $N_{m_i}^0$ of hypersurfaces\linebreak
	$N_{m_i} \subseteq F^{-1}(-\varepsilon_{m_i}) \cup F^{-1}(\varepsilon_{m_i})$, $i \in \{1, \ldots, k\}$;
	\item closures $K_{m_i}$ of components $W_{m_i}$ of the complements $M^n \setminus N_{m_i}^0$, $i \in \{1, \ldots, k\}$.
\end{itemize}
Assume that these objects are interconnected by the correlations
\begin{gather*}
B_{\delta_i} \supset K_{m_i} \supset W_{m_i} \ni x_0 \,, \quad i \in \{1, \ldots, k\} \,, \\
W_{m_i} \supset B_{\delta_{i+1}} \,, \quad i \in \{1, \ldots, k-1\} \,.
\end{gather*}

Let us fix $\delta_{k+1} \in (0, \delta_k/2)$, such that $B_{\delta_{k+1}} = \{ y \in \bR^n \,|\, \rho(y, x_0) \leq \delta_{k+1} \} \subset W_{m_k} \subset G$.

By repeating the argument above, we obtain $m_{k+1} > m_k$, such that $x_0 \in U^{m_{k+1}} \subset \Cl{U^{m_{k+1}}} \subset \Int B_{\delta_{k+1}}$. Then the set $N_{m_{k+1}} = \Fr U^{m_{k+1}}$ is a closed hypersurface in $\bR^n$ and $N_{m_{k+1}} \subset \Int B_{\delta_{k+1}}$. There exists also a component $N_{m_{k+1}}^0$ of this hypersurface, which comply with the following property: if a component $W_{m_{k+1}}$ of the complement $\bR^n \setminus N_{m_{k+1}}^0$ contains $x_0$, then its closure $K_{m_{k+1}}$ lies in $B_{\delta_{k+1}}$.

If we continue this process by induction, we will construct a sequence of closed connected hypersurfaces $\{ N_{m_{i}}^0,\; i \in \bN \}$, a sequence of components $\{ W_{m_i} \}$ of the sets $\bR^n \setminus N_{m_i}^0$, and a sequence of their closures $\{ K_{m_i} = \Cl{W_{m_i}} \}$, which are interconnected by the correlations
\begin{equation}\label{eq_seq_of_domains}
G \supset B_{\delta_i} \supset K_{m_i} \supset W_{m_i} \supset B_{\delta_{i+1}} \ni x_0 \,, \quad i \in \bN \,.
\end{equation}

All $B_{\delta_i}$ are closed $n$-disks, therefore all sets $K_{m_i}$ are compact. By construction hypersurfaces $N_{m_i}^0$ are boundaries of the sets $K_{m_i}$. It follows from relations~\eqref{eq_seq_of_domains} that $K_{m_i} \supset K_{m_j} \ni x_0$ when $i < j$, so $x_0 \in \bigcap_{i \in \bN} K_{m_{i}}$. It is obvious that $\lim_{i \to \infty} \delta_i = 0$.
Therefore, the family of sets $\{B_{\delta_i}\}$ forms the basis of neighbourhoods of the point $x_0$ in $\bR^n$. So, $\{x_0\} = \bigcap_{i \in \bN} K_{m_{i}}$. Finally, hypersurfaces $N_{m_i}^0$ are disjoint, since $N_{m_i}^0 \subset F^{-1}(-\varepsilon_{m_i}) \cup F^{-1}(\varepsilon_{m_i})$ and by construction $|\varepsilon_{m_i}| \neq |\varepsilon_{m_j}|$ for $i \neq j$.

Thus, the family $\{H_i^{n-1} = N_{m_i}^0\}_{i \in \bN}$ of connected closed hypersurfaces meets all conditions of
Definition~\ref{defn_L_function}.

\medskip

Theorem is proved.





\bibliographystyle{model1-num-names}
\bibliography{V-P-Sh-en-4}







\end{document}